\documentclass[10pt]{article} 

\usepackage{hyperref} 

\usepackage{amssymb} 

\usepackage{amsfonts}

\usepackage{stmaryrd}

\usepackage{amsmath} 

\usepackage{amsthm} 

\usepackage{epsfig}
\usepackage{graphicx}

\usepackage{dsfont}

\usepackage{eepic} 
\usepackage{mathrsfs}
\def\PP-{\discretionary{-}{-}{-}}

\newtheorem{theorem}{Theorem}[section]

\newtheorem{lemma}[theorem]{Lemma}

\newtheorem{maintheorem}{Theorem}

\author{M\'ario Bessa\thanks{Partially supported by Funda\c c\~ao para a Ci\^encia e a Tecnologia (FCT), SFRH/BPD/20890/2004.}\,\, and Maria Carvalho}
\begin{document}

\title{Imperfect friezes of integers}

\maketitle

\begin{abstract}
\noindent We show that for any positive forward density subset
$\mathscr{N}\subset\mathbb{Z}$, there exists $N\in\mathscr{N}$, such
that, for all $n\geq N$, $\mathscr{N}$ contains almost perfect
$n$-scaled reproductions of any previously chosen finite set of integers.
\end{abstract}

\section{Introduction}

Many problems in Number Theory are easy to state but very difficult
to solve. A quintessential example is the yet unsolved famous
\emph{Goldbach's conjecture} which asserts that all integers greater
than or equal to $4$ can be written as the sum of two primes.
Another renowned problem, aiming to find highly symmetric and
arbitrarily long blocks of equidistant points within a given subset
of the integers, is to settle whether the celebrated set of primes
contains arithmetic progressions with arbitrarily large size.

We say that a set $ \mathscr{N} \subset \mathbb{Z}$ has positive
\emph{density} in $\mathbb{Z}$ if

\begin{equation}\label{densidade}
\Delta(\mathscr{N}):=\underset{n\rightarrow
+\infty}{\lim}\frac{1}{2n+1}\#\{-n\leq i\leq n\colon
i\in\mathscr{N}\}>0,
\end{equation}
where $\#$ denotes the set cardinal. The upper (resp. lower) density
is defined analogously by taking the $\limsup$ (resp. $\liminf$).
For instance, $\Delta(\mathbb{Z})=1$, $\Delta(\mathcal{F})=0$ if
$\mathcal{F}$ is a finite set and $\Delta(\{\text{Odd
integers}\})=\Delta(\{\text{Even integers}\})=1/2$.

Szemer\'edi \cite{S} proved that any positive upper density subset
$\mathscr{N}\subset\mathbb{Z}$ contains arbitrarily long arithmetic
progressions. Unfortunately, we cannot apply Szemer\'edi's theorem
to the set of primes because its density is zero.\footnote{Denote
the number of primes less than $N$ by $\pi(N)$. Recall the
asymptotic relation $\pi(N)\sim \frac{N}{\ln(N)}$ and deduce that
$\Delta(\{\text{Primes}\})=0$.} This question was addressed recently
by Ben Green and the Fields Medal winner Terence Tao, and solved
positively in the remarkable work \cite{GT}.

Szemer\'edi's theorem guarantees that, taking
$\mathscr{N}\subset\mathbb{Z}$ with positive upper density and an
integer $k\geq 1$, there exist $a,b\in\mathbb{Z}$ such that
$a+jb\in\mathscr{N}$, for $j=0,...,k-1$. However, this result does
not give any information about the common difference $b$. In
particular, we may ask if $\mathscr{N}$ contains a finite arithmetic
progression with common difference equal to a previously fixed
$d\in\mathbb{N}$ but, in general, this is false (e.g. $d \text{ an
odd integer}$ and $\mathscr{N}=\{\text{Even integers}\}$). Let us
see how we overcame this difficulty.

Appoint $\mathscr{N}\subset\mathbb{Z}$, then take $k\in\mathbb{N}$
and consider a finite set $\mathcal{Q}=\{q_1,q_2,...,q_k\}$ of
$\mathbb{Z}$ such that $q_1<q_2<...<q_k$. Fix $n\in\mathbb{N}$,
bigger or equal to $k$, and $\epsilon>0$. An \emph{$n$-scale of
$\mathcal{Q}$ $\epsilon$-contained in $\mathscr{N}$} is a set
$\{r_1,r_2,...,r_k\}\in\mathscr{N}$ with $k$ elements such that
$$|r_i-\tau_i|<n\epsilon,$$
where $\tau_1=0$ and, if $k > 1$,
$$\tau_2=\frac{q_2-q_1}{q_k-q_1}n,\,\,...\,\,,\tau_{k-1}=\frac{q_{k-1}-q_1}{q_{k}-q_1}n,\,\,\tau_k=n.$$
Observe that $r_i-r_{i-1}$, when normalized by the size of the
interval $[q_1,q_k]$, is an $n$-homothety of $q_i-q_{i-1}$ up to an
error not exceeding $2n\epsilon$.

We will see that, under a sharper definition of density of
$\mathscr{N}$, \emph{$\epsilon$-contained $n$-scale sequences} exist
in $\mathscr{N}$ for any $\epsilon>0$ and any large enough $n$
depending on $\mathscr{N},$ on the fixed set $\mathcal{Q}$ and on
the required accuracy $\epsilon$. Moreover, this result holds for
any finite subset of positive integers, not necessarily within an
arithmetic progression.

In the sequel, we will say that $\mathscr{N}\subset\mathbb{Z}$ has
\emph{positive forward density} if the following limit exists and is
positive:

\begin{equation*}
\Delta^+(\mathscr{N}):=\underset{n\rightarrow +\infty}{\lim}\frac{1}{n+1}\#\{0\leq i\leq n\colon i\in\mathscr{N}\}.
\end{equation*}

\bigskip

\begin{maintheorem}\label{main}
If $\mathscr{N}\subset\mathbb{Z}$ has positive forward density,
given $\epsilon > 0$, $k\in\mathbb{N}$ and $\mathcal{Q}$ any set of
$k$ integers, there exists $N\in\mathbb{N}$ such that, for all
$n\geq N,$ we can find an $n$-scale of $\mathcal{Q}$
$\epsilon$-contained in $\mathscr{N}$.
\end{maintheorem}

\section{Proof of Theorem~\ref{main}}

Let $X=\{0,1\}^{\mathbb{Z}}$ be the space of sequences of $0$'s and
$1$'s. We define the shift map $\sigma\colon X\rightarrow X$ by
$$\sigma(...x_{-2}\,x_{-1}\,\underline{x_{0}}\,x_{1}\,x_{2}...)=...x_{-1}\,x_{0}\,\underline{x_{1}}\,x_{2}\,x_{3}...$$
For example, $\sigma(...000000...=\overline{0})=\overline{0}$ is a
fixed point of $\sigma$; the periodic sequence $\overline{10}$
defined by $y_{2m}=1$ and $y_{2m+1}=0$, for $m\in\mathbb{Z}$, is
fixed by $\sigma \circ \sigma$.

\medskip

Given $\mathscr{N}\subset\mathbb{Z}$, we can single out in $X$ a
unique sequence $(x_m)_{m\in\mathbb{Z}}$ which \emph{detects} if an
integer belongs to $\mathscr{N}$: $x_m=1$ if $m\in\mathscr{N}$ and
$x_m=0$ if $m\notin\mathscr{N}$. For example, if
$\mathscr{N}=\{\text{Even numbers}\}$, then
$(x_m)_{m\in\mathbb{Z}}=\overline{10}$. We will call
$(x_m)_{m\in\mathbb{Z}}$ \emph{the sequence that observes}
$\mathscr{N}$.

\medskip

Let $\mathscr{N}\subset\mathbb{Z}$ be a positive forward density set
and $(x_m)_{m\in\mathbb{Z}}$ the sequence that observes it. Consider
$$\Gamma:=\{(y_m)_{m\in\mathbb{Z}}\in X\colon y_0=1\}.$$

\bigskip

\begin{lemma}

$$\underset{n\rightarrow\infty}{\lim}\frac{1}{n}\sum_{i=0}^{n-1}\mathds{1}_{\Gamma}(\sigma^i((x_m)_{m\in\mathbb{Z}}))>0.$$

\end{lemma}

\begin{proof}

Notice that
$\mathds{1}_{\Gamma}(\sigma^i((x_m)_{m\in\mathbb{Z}}))=1$ if
$\sigma^i((x_m)_{m\in\mathbb{Z}})\in \Gamma$, that is, if $i\in
\mathscr{N};$ and we have 
$\mathds{1}_{\Gamma}(\sigma^i((x_m)_{m\in\mathbb{Z}}))=0$ otherwise.
Therefore,
\begin{eqnarray*}
\underset{n\rightarrow\infty}{\lim}\frac{1}{n}\sum_{i=0}^{n-1}\mathds{1}_{\Gamma}(\sigma^i((x_m)_{m\in\mathbb{Z}})
&=&\underset{n\rightarrow +\infty}{\lim}\frac{1}{n}\#\{0\leq i\leq n-1\colon i\in\mathscr{N}\}\\
&=&\Delta^+(\mathscr{N})>0.
\end{eqnarray*}

\end{proof}

\bigskip

For simplicity of notation, let $\beta$ be the positive limit
$$\beta=\underset{n\rightarrow\infty}{\lim}\frac{1}{n}\sum_{i=0}^{n-1}\mathds{1}_{\Gamma}(\sigma^i((x_m)_{m\in\mathbb{Z}})).$$
Fix now the accuracy $\epsilon>0$ required in Theorem~\ref{main} and
consider $\overline{\epsilon}= \min\{\epsilon, 1\}$ if $k=1$, and
$\overline{\epsilon}< \min\{\epsilon,\frac{1}{2(q_k-q_1)}\}$ if
$k>1.$ Notice that, this way, $\overline{\epsilon} < 2$ because, if
$k \geq 1$, then $q_k$ and $q_1$ belong to $\mathds{Z}$. Then take
$\delta\in \, ]0,\beta[$ verifying
\begin{equation}\label{first}
\frac{\beta+\delta}{\beta-\delta}<1+\frac{\overline{\epsilon}}{2}
\end{equation}
and select $n_0\in\mathbb{N}$ such that, for every $n\geq n_0$, we
have
\begin{equation}\label{second}
\left|\beta-\frac{1}{n}\sum_{i=0}^{n-1}\mathds{1}_{\Gamma}(\sigma^i((x_m)_{m\in\mathbb{Z}}))\right|<\delta.
\end{equation}
Observe that, this way,
\begin{equation}\label{beta}
\beta-\delta<\frac{1}{n_0}\sum_{i=0}^{n_0-1}\mathds{1}_{\Gamma}(\sigma^i((x_m)_{m\in\mathbb{Z}}))\leq
1.
\end{equation}

Moreover, choose an integer $N_0$ satisfying the inequality
\begin{equation}\label{third}
 N_0>\max\left\{\frac{2
 n_0}{\overline{\epsilon}\,(\beta-\delta)},\frac{4}{\overline{\epsilon}}\right\}
\end{equation}
which implies that $N_0 > n_0$ because $\beta-\delta < 1$ and
$\overline{\epsilon}<2$. Then:

\bigskip

\begin{lemma}\label{Jairo}
For all $n\geq N_0$ and all $t\in[0,1]$ there exists
$r\in\{0,1,...,n\}$ such that:
\begin{enumerate}
\item [(i)] $\sigma^r((x_m)_{m\in\mathbb{Z}}))\in\Gamma;$
\item [(ii)] $|\frac{r}{n}-t|<\overline{\epsilon}$.
\end{enumerate}
\end{lemma}

\bigskip

\begin{proof} The following argument was suggested by the proof of Lemma 3.12 of ~\cite{B}. Let us assume, by contradiction, that there exist $n\geq N_0$ and
$t\in[0,1]$ such that $\sigma^{r}((x_m)_{m\in\mathbb{Z}})\notin
\Gamma$ for all  $r\in\{0,1,...,n\}$; in particular, this holds for
all $r \in \, ]n(t-\overline{\epsilon}),n(t+\overline{\epsilon})[$.

\bigskip

Let $[s_1,s_2]$ be the maximal closed interval in
$]n(t-\overline{\epsilon}),n(t+\overline{\epsilon})[\, \cap \,
[0,n]$ where $s_1, s_2\in\mathbb{Z}$.

\bigskip

\noindent \emph{Claim}: \begin{equation}\label{four}
 s_2-s_1> \frac{n\overline{\epsilon}}{2}.
\end{equation}

\bigskip

\noindent In fact:

\begin{itemize}
\item If $n(t-\overline{\epsilon})\geq 0$ and $n(t+\overline{\epsilon}) \leq n$, then,
denoting by $\lfloor z \rfloor$ the biggest integer less or equal
than $z$, we have
$$s_2-s_1 \geq \lfloor n(t+\overline{\epsilon})\rfloor - (\lfloor n(t-\overline{\epsilon})\rfloor + 1)
> n(t+\overline{\epsilon}) - 1 -
n(t-\overline{\epsilon})-1=2n\overline{\epsilon} - 2
> \frac{n\overline{\epsilon}}{2}$$
since $n \geq N_0$ and, by (\ref{third}),
$N_0\overline{\epsilon}>4.$

\item If $n(t-\overline{\epsilon})\geq 0$ and $n(t+\overline{\epsilon}) > n$, then
$$s_2-s_1 \geq n - (\lfloor n(t-\overline{\epsilon})\rfloor + 1)
\geq n - n(t-\overline{\epsilon})-1=n\overline{\epsilon}+(n-1-nt)
\geq n\overline{\epsilon} - 1
> \frac{n\overline{\epsilon}}{2}.$$

\item If $n(t-\overline{\epsilon})< 0$ and $n(t+\overline{\epsilon}) \leq n$, then
$$s_2-s_1 \geq \lfloor n(t+\overline{\epsilon})\rfloor
> n(t+\overline{\epsilon}) -1 > n\overline{\epsilon}-1 > \frac{n\overline{\epsilon}}{2}.$$

\item If $n(t-\overline{\epsilon})< 0$ and $n(t+\overline{\epsilon}) > n$, then
$s_2-s_1=n > \frac{n\overline{\epsilon}}{2}$ since
$\overline{\epsilon} < 2.\,\,\square$

\end{itemize}

\bigskip

Let us go back to the proof of the Lemma. If $s_1\geq n_0$, then,
since $s_1\leq n$ and
$\sigma^{r}(((x_m)_{m\in\mathbb{Z}}))\notin\Gamma$ for all
$r\in\{0,1,...,n\}$, we may deduce that
\begin{eqnarray*}
\beta-\delta&\overset{(\ref{second})}{<}& \frac{1}{s_2}\sum_{i=0}^{s_2-1}\mathds{1}_{\Gamma}(\sigma^i((x_m)_{m\in\mathbb{Z}}))=\frac{1}{s_2}\sum_{i=0}^{s_1-1}\mathds{1}_{\Gamma}(\sigma^i((x_m)_{m\in\mathbb{Z}}))\\
&\overset{(\ref{four})}{\leq} &\frac{1}{s_1+\frac{n\overline{\epsilon}}{2}}\sum_{i=0}^{s_1-1}\mathds{1}_{\Gamma}(\sigma^i((x_m)_{m\in\mathbb{Z}}))\leq \frac{1}{s_1(1+\frac{\overline{\epsilon}}{2})}\sum_{i=0}^{s_1-1}\mathds{1}_{\Gamma}(\sigma^i((x_m)_{m\in\mathbb{Z}}))\\
&\overset{(\ref{second})}{<}&\frac{\beta+\delta}{1+\frac{\overline{\epsilon}}{2}}\overset{(\ref{first})}{<}\beta-\delta,
\end{eqnarray*}
which is a contradiction.

On the other hand, if $s_1< n_0$, then, since $s_1 \geq 0$ and
$n\geq N_0$, we have
\begin{equation}\label{five}
s_2 \geq
s_2-s_1\overset{(\ref{four})}{>}\frac{n\overline{\epsilon}}{2}\overset{(\ref{third})}{>}\frac{n_0}{\beta-\delta}>n_0,
\end{equation}
because $\beta-\delta\overset{(\ref{beta})}{<} 1.$ Therefore
\begin{eqnarray*}
\beta-\delta&\overset{(\ref{second})}{<}&  \frac{1}{s_2}\sum_{i=0}^{s_2-1}\mathds{1}_{\Gamma}(\sigma^i((x_m)_{m\in\mathbb{Z}}))= \frac{1}{s_2}\sum_{i=0}^{s_1-1}\mathds{1}_{\Gamma}(\sigma^i((x_m)_{m\in\mathbb{Z}}))\overset{(\ref{four})}{<} \frac{s_1}{s_1+\frac{n\overline{\epsilon}}{2}}\\
&<&
\frac{n_0}{\frac{n\overline{\epsilon}}{2}}\overset{(\ref{third})}{<}\beta-\delta,
\end{eqnarray*}
which is again a contradiction.

\end{proof}

\bigskip

We may now end the proof of Theorem~\ref{main}. Given $k\in
\mathds{N}$ and $\mathcal{Q}:=\{q_1,...,q_k\}\subset\mathbb{Z}$ such
that $q_1<q_2<...<q_k$, we fix $N_0$ as above and $n\geq
N=\max\{k,N_0\}$. Then we apply $k$ times the Lemma~\ref{Jairo},
using the numbers $t_1=0$, and, if $k>1$,

$$t_2=\frac{q_2-q_1}{q_k-q_1},\,\,...\,\,,t_{k-1}=\frac{q_{k-1}-q_1}{q_k-q_1},\,\,t_k=1.$$
This way we get, for $(x_m)_{m\in\mathbb{Z}}$, a finite set $\{r_1,
..., r_k\} \subset \{0,1,...,n\}$ such that:
\begin{enumerate}
\item [(i)] $\sigma^{r_i}((x_m)_{m\in\mathbb{Z}})\in\Gamma;$
\item [(ii)] $|\frac{r_i}{n}-t_i|<\overline{\epsilon}$.
\end{enumerate}
Item (i) means that, for all $i=1,...,k,$

$$x_{r_i}=1, \text{ for all }i=1,...,k,$$
which is equivalent to say that

$$r_i\in\mathscr{N}, \text{ for all }i=1,...,k.$$

\bigskip

\noindent Besides, if $k>1$, then $r_i \not= r_j$ if $i \not= j$.
Indeed, by item (ii), for each $i \in \{1,...,k-1\}$, we have
$$n\frac{q_i-q_1}{q_k-q_1} - n\overline{\epsilon} < r_i <n\frac{q_{i}-q_1}{q_k-q_1} + n\overline{\epsilon}$$
$$n\frac{q_{i+1}-q_1}{q_k-q_1} - n\overline{\epsilon} < r_{i+1} < n\frac{q_{i+1}-q_1}{q_k-q_1} +
n\overline{\epsilon}$$ and

\bigskip

\noindent \emph{Claim:}
$$n\frac{q_{i+1}-q_1}{q_k-q_1} - n\overline{\epsilon} > n\frac{q_{i}-q_1}{q_k-q_1} + n\overline{\epsilon}.$$

\noindent This means that the intervals where $r_i$ and $r_{i+1}$
live are disjoint, and therefore these numbers cannot be equal.

\bigskip

The last inequality is a consequence of the choice
$\overline{\epsilon}\leq \frac{1}{2(q_k-q_1)}$. In fact, taking into
account that $\mathcal{Q}\subset\mathbb{Z}$, from it, we get:
$$2n\overline{\epsilon} < n\left(\frac{1}{q_k-q_1}\right) \leq n\left(\frac{q_{i+1}-q_i}{q_k-q_1}\right)= n\left(\frac{q_{i+1}-q_1}{q_k-q_1}- \frac{q_i-q_1}{q_k-q_1}\right)$$
as wished. $\square$

\bigskip
Finally, $\{r_i\}_{i=1}^k$ is an $n$-scale of $\mathcal{Q}$
$\epsilon$-contained in $\mathscr{N}$ due to the inequality
$\overline{\epsilon}<\epsilon$, the judicious choice of the $t_i$'s
and item (ii). $\square$

\bigskip

\flushleft

{\bf M\'ario Bessa} \ \  (bessa@fc.up.pt)\ \ FCUP, Rua do Campo
Alegre, 687, 4169-007 Porto, Portugal \emph{ and } ESTGOH - IPC, Rua
General Santos Costa,\ 3400-124, Oliveira do Hospital, Portugal.

{\bf Maria Carvalho} \ \  (mpcarval@fc.up.pt)\ \ CMUP, Rua do Campo
Alegre, 687, 4169-007 Porto, Portugal

\end{document}